\documentclass[12pt]{article}
\usepackage{amssymb,amsfonts,url}
\newtheorem{theorem}{Theorem}[section]
\newtheorem{lemma}[theorem]{Lemma}
\newtheorem{prop}[theorem]{Proposition}
\newtheorem{cor}[theorem]{Corollary}
\newcommand{\qed}{$\Box$}
\newtheorem{unnumber}{}

\newtheorem{prepf}[unnumber]{Proof}
\newenvironment{proof}{\prepf\rm}{\endprepf}

\begin{document}

\title{The power graph of a torsion-free group\footnote{The second and third
author acknowledge funding from the School of Mathematics and Statistics for
summer internships during which this research was carried out.}}
\author{Peter J. Cameron\footnote{Corresponding author: pjc20@st-andrews.ac.uk},
Horacio Guerra\footnote{Current address: School of Mathematics and Statistics, Newcastle upon Tyne NE1 7RU, UK} and \v{S}imon Jurina\\
\small{University of St Andrews, North Haugh, St Andrews, Fife KY16 9SS, UK}}
\date{}
\maketitle

\begin{abstract}
The \emph{power graph} $P(G)$ of a group $G$ is the graph whose vertex set is
$G$, with $x$ and $y$ joined if one is a power of the other; the
\emph{directed power graph} $\vec{P}(G)$ has the same vertex set, with an arc
from $x$ to $y$ if $y$ is a power of $x$. It is known that, for finite groups,
the power graph determines the directed power graph up to isomorphism. However,
it is not true that any isomorphism between power graphs induces an isomorphism
between directed power graphs. Moreover, for infinite groups the power graph
may fail to determine the directed power graph.

In this paper, we consider power graphs of torsion-free groups. Our main 
results are that, for torsion-free nilpotent groups of class at most~$2$,
and for groups in which every non-identity element lies in a unique maximal
cyclic subgroup,
the power graph determines the directed power graph up to isomorphism. For
specific groups such as $\mathbb{Z}$ and $\mathbb{Q}$, we obtain more precise
results. Any isomorphism $P(\mathbb{Z})\to P(G)$ preserves orientation, so
induces an isomorphism between directed power graphs; in the case of
$\mathbb{Q}$, the orientations are either all preserved or all reversed.

We also obtain results about groups in which every element is contained in a
unique maximal cyclic subgroup (this class includes the free and free abelian
groups), and about subgroups of the additive group of $\mathbb{Q}$ and about
$\mathbb{Q}^n$.

MSC: 05C25, 20F99
\end{abstract}

\section{Introduction}

Let $G$ be a group. Then the \emph{power graph} of $G$ is the graph $P(G)$
with vertex set $V(P(G)) =G $ and edge set
\[E(P(G)) = \{\{x,y\} : (\exists\,n\in\mathbb{Z}\setminus\{0\})\,
(x = y^n\hbox{ or }y = x^n)\}.\]

The \emph{directed power graph} of $G$ is the directed graph $\vec{P}(G)$ with
vertex set $V(\vec{P}(G)) =G $ and arc set
\[E(\vec{P}(G)) = \{(x,y):(\exists\,n\in\mathbb{Z}\setminus\{0\})\,
(y = x^n)\}.\]
Thus, $\vec{P}(G)$ is an orientation of $P(G)$.

When $x$ and $y$ are connected in $P(G)$, we write $x \sim y$. If $a$ is a
power of $b$ in $G$ we denote this by $b \to a$.

In a graph $\Gamma$, we denote the set of neighbours of a vertex $x$ by
$N(x)$; in a directed graph, we denote the set of in-neighbours of $x$ by 
$I(x)$, and the set of out-neighbours by $O(x)$. Since we will always be
considering power graphs of groups, we denote $N_G(x)$ for the set of
neighbours of $x$ in $P(G)$, and $I_G(x)$, $O_G(x)$ for the sets of
in- and out-neighbours of $x$ in $\vec{P}(G)$.

\smallskip

The directed power graph was first defined in the context of semigroups by
Kelarev and Quinn~\cite{kq}; the undirected power graph by Chakrabarty 
\emph{et al.}~\cite{cgs}. The first author~\cite{pjc} showed the following
result:

\begin{theorem}
Let $G$ and $H$ be finite groups such that $P(G)\cong P(H)$. Then
$\vec{P}(G)\cong\vec{P}(H)$.
\end{theorem}

However, it is not true that any isomorphism from $P(G)$ to $P(H)$ preserves
orientations of edges; and the theorem above fails for infinite groups.

\paragraph{Example} Let $G$ be the cyclic group of order $6$, generated by
$a$. Then $P(G)$ is the complete graph $K_6$ with the two edges $\{a^2,a^3\}$
and $\{a^4,a^3\}$ removed. So its automorphism group permutes $\{1,a,a^5\}$
transitively; the power graph does not determine the identity uniquely.

\paragraph{Example} (from~\cite{cg}) Let $G$ be the \emph{Pr\"ufer group}
$\mathbb{Z}_{p^\infty}$, defined as the quotient $\mathbb{L}_p/\mathbb{Z}$,
where $\mathbb{L}_p$ is the set of rationals with $p$-power denominators
(where $p$ is prime). Then every element of $G$ has $p$-power order, and every
proper subgroup is a finite cyclic group; so $P(G)$ is a countable complete
graph. Thus, knowledge of $P(G)$ does not even determine the prime $p$.

\medskip

In this paper we consider torsion-free groups, to avoid difficulties suggested
by the above examples. Our main results are the following theorems.

\begin{theorem}\label{t0}
Let $H$ be a group with $P(H)$ isomorphic to $P(\mathbb{Z})$. Then $H$ is
isomorphic to $\mathbb{Z}$, and any isomorphism from $P(\mathbb{Z})$ to $P(H)$
induces an isomorphism from $\vec{P}(\mathbb{Z})$ to $\vec{P}(H)$.
\end{theorem}

\begin{theorem}\label{t1}
Let $G$ and $H$ be nilpotency class $2$ torsion-free groups. Then
$P(G)\cong P(H)$ implies $\vec{P}(G)\cong\vec{P}(H)$.
\end{theorem}

We will see examples showing that, under these hypotheses, it is not true
that $P(G)\cong P(H)$ implies $G\cong H$; and also, examples to show that
some hypothesis on $G$ is needed.

\begin{theorem}\label{t2}
Let $G$ be a torsion-free group in which any non-identity element lies in a
unique maximal cyclic subgroup. Then, for any group $H$, any isomorphism
from $P(G)$ to $P(H)$ induces an isomorphism from $\vec{P}(G)$ to $\vec{P}(H)$.
\end{theorem}

The class of groups covered by this theorem include direct sums of copies of
the additive group of $\mathbb{Z}$, free groups, and indeed
any torsion-free nilpotent group of class~$2$ in which every element is 
contained in some maximal cyclic subgroup (see Proposition~\ref{torsion}).
Our result about such groups is stronger than indicated: see
Theorem~\ref{uniquecyclic}.

\begin{theorem}\label{t3}
Let $\mathbb{Q}$ be the additive group of rational numbers, and
$G=\mathbb{Q}^n$. Then, for a group $H$, if $P(G)\cong P(H)$, then
$\vec{P}(G)\cong\vec{P}(H)$. Moreover, if $n=1$, then any isomorphism from
$P(G)$ to $P(H)$ either preserves or reverses the orientation of
edges.
\end{theorem}

We also include some detailed results about power graphs of subgroups of
$\mathbb{Q}$, where we exhibit non-isomorphic subgroups with isomorphic power
graphs.

We note that further investigation of the power graph can be found in
\cite{aacns}; this and other papers concentrate on graph-theoretic aspects.

\section{On the definition}

We have excluded $n=0$ in the definition of edges in the power graph and the
directed power graph. We make some brief comments on this. The alternative
definition would ensure that there is an edge from every vertex $x$ to $x^0=1$
in the directed power graph.

If $G$ is a torsion group (in particular, if $G$ is finite), then this makes
no difference at all, since our definition as stated gives an edge from 
$x$ to $1$ if $x^n=1$ for some $n>0$. 

We are concerned here with torsion-free groups. A group $G$ is
\emph{torsion-free} if every non-identity element has infinite order. Note
that, if $G$ is torsion-free and $a\in G$ satisfies $a^m=a^n$ where $m\ne n$,
then $a^{m-n}=1$, so $a=1$.

For a torsion-free group, with the 
definition modified to allow $n=0$, the identity is the unique sink in the
directed power graph (there is an arc from every vertex to it), and so we can
uniquely identify it. The situation in the undirected power graph is a little
different:

\begin{prop}
Let $x$ be an element of the torsion-free group $G$. Then $x$ is joined to
every vertex in the undirected power graph (including edges from $x$ to $x^0$)
if and only if one of the following holds:
\begin{itemize}\itemsep0pt
\item $x$ is the identity;
\item $G$ is the infinite cyclic group and $x$ is a generator.
\end{itemize}
\end{prop}

\begin{proof}
The sufficiency is clear. So suppose that, for every $y\in G$, either $y=x^n$
or $x=y^n$ for some integer $n$, but $x$ is neither the identity nor a 
generator. Note that $n$ is unique in either case. Since $x$ is not a 
generator, there exists $y$ such that $x=y^n$ for some $n>1$. Choose an
integer $m>1$ coprime to $n$ and let $z=y^m$.  If $x=z^k$, then $y^n=y^{mk}$,
so $n=mk$, implying that $m\mid n$; if $z=x^k$, then $y^{nk}=y^m$, whence
$nk=m$, and $n\mid m$. Either statement contradicts $\gcd(n,m)=1$. \qed
\end{proof}

Thus, if $G$ is not infinite cyclic, we can recognise the identity. In the case
where $G=\mathbb{Z}$, the identity and the two generators are indistinguishable
in the power graph, and are permuted transitively by its automorphism group,
so we can choose any one to be the identity and delete the edges containing it
to get a graph isomorphic to the power graph as defined in this paper. 

So our theorems would be essentially unaffected by changing the definition.
We use the definition given because it makes some of the arguments simpler.

\section{Preliminary results}

We collect here a few lemmas of general use.

\begin{lemma}\label{order}
Let $G$ be a group with $P(G)$ having exactly one isolated vertex. Then $G$ is
torsion-free.
\end{lemma}

\begin{proof}
Since $P(G)$ has a unique isolated vertex, it suffices then to show that this
must be the identity of $G$. Let $a \in G$ be non-identity. If $a = a^{-1}$,
we have $a^2 = 1_G$, so $a$ is not the isolated vertex. On the other hand, if
$a \neq a^{-1}$, then $a$ and $a^{-1}$ are joined, so again $a$ is not isolated.
\qed
\end{proof}

For $a, b \in G$ define
\[ S_{a, b} := \{c \in G: c \sim b \hbox{ and }c \not\sim a\}.\]

\begin{lemma}\label{inverse}
Let $G$ be a group with $P(G)$ having exactly one isolated vertex, and suppose
that $a,b\in G$ with $a \sim b$. Then
$S_{a,b} = S_{b,a} = \varnothing$ if and only if $a = b^{\pm 1}$.
\end{lemma}

\begin{proof}
($\Leftarrow$) Observe that $a= b^{\pm 1}$ implies that $a$ and $b$ have the
same neighbours in the power graph.

($\Rightarrow$) We have that $a \sim b$. If either is the identity then so is
the other, so we are done. Hence we have that $a$ and $b$ are non-identity.

Suppose first that $a = b^m$ for some $m \in \mathbb{Z}$.
If $|m| > 1$, then choose $j > 1$
such that $\gcd(j, m) =1$. We claim $b^j \not\sim a$. Indeed, suppose that
$a = b^{jt}$, for some $t \in \mathbb{Z}$. Then $b^m = b^{jt}$, so by
Lemma~\ref{order} and our earlier remark, we deduce that $m = jt$, so
$\gcd(j,m) = j > 1$, a contradiction.

Otherwise, suppose $b^j = a^t$ for some $t \in \mathbb{Z}$. We deduce that
$b^j = b^{mt}$, so by Lemma~\ref{order} we have $j = mt$, so
$\gcd(j,m) = |m| > 1$, also a contradiction. 

Thus, $b^j \not\sim a$, and since there are infinitely many choices for $j$, all giving pairwise distinct elements $b^j$, we have that $S_{a,b}$ is infinite.
Therefore we must have $|m| = 1$, so $a = b^{\pm 1}$.

Similarly, if $a^m = b$
for some $m \in \mathbb{Z}$, a symmetric argument shows $ m = \pm 1$. \qed
\end{proof}

Note that the proof shows that, if $b\to a$ but $a\not\to b$, then $S_{a,b}$ is
infinite.

\medskip

From now on, when dealing with a torsion-free group, we will use without
mention the result above: we can always recognise inverse elements in the
power graph of the group.

We conclude this section with a couple more results which will be needed later.

\begin{lemma}\label{Q}
Let $G$ be a torsion-free group and $x$ a non-identity element of $G$. Then
the induced subgraph of $P(G)'$ on $O(x)=O_G(x)$ is a connected
subgraph of $P(G)'$, and there are no edges between 
$I(x)$ and $O(x)$ in $P(G)'$.
\end{lemma}

\begin{proof}
Let $x^m$ and $x^n$ be in $O(x)$. Then if $p$ is a prime dividing neither $m$
nor $n$, then $x^p$ is joined to both $x^m$ and $x^n$ in the complement of the
power graph.

Finally, if we have $y \to x \to z$, then $y \to z$. Thus, no in-neighbour is 
connected to an out-neighbour in the complement of the power graph. \qed
\end{proof}

\begin{lemma}\label{isom_nbrs}
Let $G$ be a torsion-free group and $H$ be a group with $P(G) \cong P(H)$.
Fix $z \in G$ such that $z \neq 1_G$ and let $f$ be an isomorphism
$f: P(G) \to P(H)$. Then $f$ induces an isomorphism from each connected
component of $P(G)'$ in $N_G(z)$ to a connected component of $P(H)'$ in
$N_H(f(z))$. Furthermore, $I_G(z) \cong I_H(f(z))$ and $O_G(z) \cong O_H(f(z))$.
\end{lemma}

\begin{proof}
By Lemma~\ref{order}, $H$ is torsion-free and $f(z) \neq 1_H$. The lemma will
follow from the next result:

\paragraph{Claim} For all connected components $C$ of $N(z)'$ and $x,y \in C$,
$f(x)$ and $f(y)$ belong to the same connected component $D$ of $N(f(z))'$.
(Here we usee $N(z)'$ for the induced subgraph of $P(G)'$ on $N(z)$.)

To verify the claim, suppose that there exists a connected component $C$ of
$N(z)'$ and $x,y \in C$ such that $f(x) \in D_1$ and $f(y) \in D_2$, where
$D_1$ and $D_2$ are different connected components of $N(f(z))'$. Then there
exists a path $(x_0=x,x_1,\ldots,x_n=y)$ in $C$. As $f$ is an isomorphism from
$P(G)$ to $P(H)$, it follows that there is a path
$(f(x_0=x),f(x_1),\ldots,f(x_n=y))$ in $P(H)'$ and $f(N(z))=N(f(z))$. Hence,
for all $i \in \{0,1,\ldots,n\}$, we have $f(x_i) \in N(f(z))'$. But this is a
contradiction as $f(x)$ and $f(y)$  belong to different connected components
of $N(f(z))'$.

We are ready to prove the lemma. As $G$ and $H$ are torsion-free, we have
$O(x)=\{x^n:n\in\mathbb{Z}\setminus\{0\}\}$, and so
$O(f(x))\cong P(\mathbb{Z})\setminus\{0\}\cong O(x)$. Furthermore, as $f$ is
an isomorphism, using our claim, we deduce that $f$ induces an isomorphism from
each connected component of $N(z)'$ to a connected component of $N(f(z))'$. It
remains to show that $I(z) \cong I(f(z))$. By Lemma~\ref{Q} we have to consider
two cases, either $f(O(z))=O(f(z))$ or $f(O(z))=D$, where $D$
is a connected component of $I(f(z))'$. In the first case, $f$ induces an
isomorphism from each connected component of $I(z)'$ to a connected component of $I(f(z))'$. In the second case, using our results, there exists a connected
component $C$ of $I(z)'$ such that $f(C)=O(f(z))'$. Hence there are two
connected components in both $N(z)'$ and $N(f(z))'$ isomorphic to
$P(\mathbb{Z})\setminus\{0\}$. Then $f$ induces an isomorphism from each of
the remaining connected components of $I(z)'$ to one of the remaining
connected components of $I(f(z))'$. Putting this together, we deduce that in
both cases $I(z) \cong I(f(z))$. \qed
\end{proof}

\section{The group $\mathbb{Z}$}

In this section we examine the power graph of $\mathbb{Z}$.

\begin{lemma}\label{finite}
Let $a, b \in \mathbb{Z}$ be such that $a \sim b$ and $a \neq \pm b$. Then 
$a \to b$ if and only if $S_{a, b}$ is finite.
\label{l41}
\end{lemma}

\begin{proof}
($\Rightarrow$) Notice that $b$ is only divisible by finitely many
$c \in \mathbb{Z}$. On the other hand, if $b$ divides $x$ then $a$ divides $x$.
Thus there are at most finitely many vertices that are connected to $b$ but
not to $a$, that is, $S_{a,b}$ is finite.

($\Leftarrow$) By Lemma~\ref{finite}, we know that if $a\not\to b$, then 
$b\to a$, so $S_{a,b}$ is infinite. \qed
\end{proof}

This shows that the undirected power graph of $\mathbb{Z}$ determines the
directed power graph, by the rule in the Lemma. Using this, we prove
Theorem~\ref{t0}.

\paragraph{Proof of Theorem~\ref{t0}} Let $G=\mathbb{Z}$. We note first that
$P(H)$ has an isolated vertex, so $H$ is torsion-free.
Putting Lemma~\ref{inverse} and Lemma~\ref{finite} together we observe that if
$a \sim b$, $a \neq b$ in $P(G)$, then one of the following holds:
  \begin{enumerate}\itemsep0pt
  \item $S_{a,b} = S_{b,a} = \varnothing$;
  \item one of $S_{a,b}$ or $S_{b,a}$ is finite and the other is infinite.
  \end{enumerate}
Therefore the same holds in $P(H)$. So consider $a \sim b$ in $P(H)$ with
$a \neq b$. If we are in the first case, then, by Lemma~\ref{inverse}, we
deduce that $a = b^{-1}$, so $\vec{P}(H)$ has directed arrows going in both
directions. But the same is true in $\vec{P}(G)$, so for all these cases the
directions agree. If we are in the second case, say $S_{a,b}$ is finite and
$S_{b,a}$ is infinite, then the corresponding elements of $G$, say $a'$ and
$b'$, have a directed edge in $\vec{P}(G)$ going from $a'$ to $b'$, but not
the other way around. Suppose that the direction in $\vec{P}(H)$ was reversed,
so that $a = b^m$ for some $m \in \mathbb{Z}$. The argument in
Lemma~\ref{inverse} shows that $ |m| > 1$ implies that $S_{a,b}$ is infinite,
contrary to our assumption. Thus we must have $a = b^{-1}$, but this is also
contrary to our assumption that $S_{b,a}$ is infinite. Thus we must have the
directions agreeing in the power graph of $G$ for the second case as well.

Now, if $G=\langle a\rangle$, then there is a directed
arrow from $a$ to every other element of $G$ except the identity. So $H$ has
such a vertex also, and $H$ is an infinite cyclic group, as needed. \qed

\paragraph{Remark} If we had used the alternative definition of the power
graph, where $x$ is joined also to $x^0=1$, then it is false that any 
isomorphism of the power graph induces an isomorphism of the directed power
graph, since as noted earlier the identity and the two generators are 
indistinguishable in the power graph. We can conclude that, with this
definition, if $P(H)\cong P(\mathbb{Z})$, then
$\vec{P}(H)\cong\vec{P}(\mathbb{Z})$.

\section{Groups with the same power graph as $\mathbb{Z}^n$}

One may be tempted to conjecture that, for all $n \in \mathbb{N}$, the power
graph of $\mathbb{Z}^n$ determines $\mathbb{Z}^n$ up to isomorphism, as we
showed was true for $\mathbb{Z}$ in Theorem~\ref{t0}. However, this is
not the case. In fact, we will prove that, for $n > 1$, all the groups
$\mathbb{Z}^n$ have isomorphic power graphs.

In this section we are interested in a wider class of groups, namely those
with the following property $(*)$:
\begin{quote}\it Every non-identity element is contained in a unique
maximal cyclic subgroup.
\end{quote}
We begin with a few remarks about this class. First,
observe that the property is equivalent to saying that the non-identity
elements of the group are partitioned into maximal cyclic subgroups. Hence:

\begin{prop}
Let $G$ be a torsion-free group in which every non-identity element is
contained in a unique maximal cyclic subgroup. Then the power graph of $G$ 
is the disjoint union of an isolated vertex (the identity) and a number of
copies of $P(\mathbb{Z})\setminus\{1_{\mathbb{Z}}\}$.
\end{prop}

How many connected components are there? This is answered by the next result.

\begin{prop}
Let $G$ be a torsion-free group in which every non-identity element lies in
a unique maximal infinite cyclic subgroup. Then the number of maximal infinite
cyclic subgroups is either $1$ or infinite. In particular, if $G$ is countable
but not isomorphic to $\mathbb{Z}$, then the number of such subgroups is
countably infinite.
\end{prop}

\begin{proof}
Suppose, for a contradiction, that $a_1,\ldots,a_k$ are all the generators for
the maximal infinite cyclic subgroups of $G$, where $k>2$. (Each such subgroup
has two generators, which are inverses of each other.) By hypothesis,
$G=\langle a_1,\ldots,a_k\rangle$.

Now $G$ acts on itself by conjugation; this action must map the set
\[ \{a_1,\ldots,a_k\} \] into itself, and so induces a subgroup of the symmetric
group $S_k$ on this set. The kernel of this action is a subgroup $H$ of finite
index in $G$ which fixes all of $a_1,\ldots,a_k$; we see that $H$ is the
centre $Z(G)$ of $G$, and so $H$ is abelian. 

Now $H$ is an infinite abelian group which is partitioned by its intersections
with the maximal cyclic subgroups of $G$. But if $a$ and $b$ are elements of
$H$ belonging to distinct such subgroups, then
$\langle a,b\rangle\cong\mathbb{Z}^2$, and this group cannot be covered by
finitely many cyclic subgroups. \qed
\end{proof}

Which groups have this property? One class is given by the next result.

\begin{prop}\label{torsion}
Let $G$ be a torsion-free group of nilpotency class $2$, and suppose $a$ and
$b$ generate distinct maximal infinite cyclic subgroups. Then
$\langle a \rangle \cap \langle b \rangle = \{1\}$.
\end{prop}

\begin{proof}
Suppose that $\langle a \rangle \cap \langle b \rangle = \langle x \rangle$,
for some $x \in G$, $x \neq 1$. So $x = a^m = b^n$ for some $n,m\in\mathbb{Z}$. 
Then, since $a^m$ is a power of $b$, we have $1 = [a^m, b]$. Since
$[a_1a_2, b] = [a_1,b] [a_2, b]$ in a nilpotent group of class $2$, we have
$1 = [a, b]^m$. But, since $G$ is torsion-free, we have $[a,b] = 1$, so
$\langle a, b\rangle$ is abelian.

So this subgroup is equal to one of $\mathbb{Z}^2$, or $\mathbb{Z} \times C_k$
for some natural number $k$. It cannot be $\mathbb{Z} \times C_k$ for any
$k>1$, as then $G$ would not be torsion-free. Also, it cannot be $\mathbb{Z}$,
since $\langle a \rangle$ and $\langle b \rangle$ are maximal cyclic, so this
would force $\langle a, b \rangle =  \langle a \rangle  = \langle b \rangle$,
contrary to our assumption. Finally, we observe that
$\langle a, b \rangle / \langle a \rangle$ is finite,
but $\mathbb{Z}^2/\langle g\rangle$ is infinite for
any $g \in \mathbb{Z}^2$. (This is clear if $g$ is the identity, so suppose
not. Let $g=(m,n)$ where, without loss of generality, $n\ne0$. The elements
$(k,0)$ for $k\in\mathbb{Z}$ all lie in different cosets of $\langle g\rangle$.)
Hence $\langle a \rangle \cap \langle b \rangle = \{1\}$, as needed. \qed
\end{proof}

This shows, for example, that $\mathbb{Z}^n$ has property $(*)$ for finite $n$.
It is enough to prove that each non-identity element of $\mathbb{Z}^n$ lies in
a maximal cyclic subgroup. The element $(a_1,\ldots,a_n)$, with $a_i$ not all
zero, lies in the maximal cyclic subgroup
$\langle(a_1/d,\ldots,a_n/d)\rangle$, where $d=\gcd(a_1,\ldots,a_n)$.

Other groups with this property include free groups.

\paragraph{Remark} Consider the two conditions on a torsion-free group $G$:
\begin{enumerate}\itemsep0pt
\item every non-identity element lies in a maximal cyclic subgroup;
\item every non-identity element lies in a unique maximal cyclic subgroup.
\end{enumerate}
Now (a) does not imply (b) in general. For take the group generated by $a$ and
$b$ with the single defining relation $a^m=b^n$ where $m,n>1$. This group is a
\textit{free product with amalgamation} $A *_C B$, where $A$ and $B$ are the
groups generated by $a$ and $b$ respectively and $C$ is generated by $a^m=b^n$. The theory of such groups tells us~\cite{bhn}:
\begin{itemize}\itemsep0pt
\item $A$ and $B$ embed into $A *_C B$;
\item any element which is not in a conjugate of $A$ or $B$ has infinite order.
\end{itemize}
It follows that the group is torsion-free. Clearly the element $a^m=b^n$ lies
in two distinct maximal cyclic subgroups. On the other hand, by
Proposition~\ref{torsion}, in torsion-free abelian groups, or nilpotent groups
of class~$2$, (a) does imply (b).

For groups with property $(*)$, we can make a strong statement about the power
graphs.

\begin{theorem}\label{uniquecyclic}
Let $G$ be a countable torsion-free group which is not cyclic, but in which
each non-identity element lies in a unique maximal cyclic subgroup. Let $H$ be
a group with $P(H)\cong P(G)$. Then
\begin{enumerate}\itemsep0pt
\item each non-identity element of $H$ lies in a unique maximal cyclic subgroup;
\item $\vec{P}(H)\cong\vec{P}(G)$;
\item any isomorphism from $P(G)$ to $P(H)$ induces an isomorphism from
$\vec{P}(G)$ to $\vec{P}(H)$.
\end{enumerate}
Moreover, all groups $G$ satisfying the hypothesis have isomorphic power graphs.
\end{theorem}

\begin{proof}
For a group satisfying the hypotheses of the
theorem, there are countably many connected components of $P(G)$ (with the
identity removed), each isomorphic to $P(\mathbb{Z})$ with the identity
removed. So the last statement holds. Since the power graph of $P(\mathbb{Z})$
determines the directions on edges (Lemma~\ref{l41}), (b) and (c) hold.

Now suppose that $f:\vec{P}(G)\to\vec P(H)$ is a directed power graph 
isomorphism. Then each connected component of $\vec{P}(H)$ has a vertex $a$
with an arc to all other vertices of the component (the image under $f$ of a
generator of a maximal cyclic subgroup of $G$); so the component together
with the identity is a maximal cyclic subgroup. This proves (a).
\end{proof}

\begin{cor}
The groups $\mathbb{Z}^n$, for $n\in\mathbb{N}$, $n\ge 2$, or the direct sum
of countably many copies of $\mathbb{Z}$,
all have isomorphic (directed) power graphs.
\end{cor}


\section{The groups $\mathbb{Q}$ and $\mathbb{Q}^n$}

Next turn to study the additive group of the rationals. Before proving the
main theorem we prove an auxiliary lemma. As before, if $a$ is a vertex of a
directed graph, let $I(a)$ and $O(a)$ denote the sets of in-neighbours and
out-neighbours of $a$.

\begin{lemma}\label{isomorphism}
For $a \in \mathbb{Q}\setminus\{0\}$, define the map
$\varphi_a: \mathbb{Q} \to \mathbb{Q}$ by $x \mapsto a^2/x$ and $0 \mapsto 0$.
Then $\varphi_a$ is an automorphism of $P(\mathbb{Q})$ and an isomorphism from
$\vec{P}(\mathbb{Q})$ to the directed power graph of $\mathbb{Q}$ with all
arrows reversed. Furthermore, it is an isomorphism from $O(a)$ to $I(a)$.
\end{lemma}

\begin{proof}
It is straightforward to verify that $\varphi_a$ is a bijection. If $x \sim y$,
then we have $x = ny$ for some $n\in\mathbb{Z}$, say. Then $a^2/x = a^2/ny$, so
$\varphi(y) =  n \varphi(x)$, so $\varphi_a(x) \sim \varphi_a(y)$. We have $ x \to y$ if, and only if, $\varphi_a(y) \to \varphi_a(x)$, if, and only if, $\varphi_a(x) \to \varphi_a(y)$ in the reversed power graph, as needed.

We have
\begin{eqnarray*}
\varphi_a(O(a)) &=& \{\varphi(n a) : n \in \mathbb{Z}\} \\
&=& \{a/n : n \in \mathbb{Z}\} \\
&=& I(a),
\end{eqnarray*}
so $\varphi_a$ maps $O(a)$ to $I(a)$ bijectively and preserves edge
relationships, as required. \qed
\end{proof}

Now we can prove part of Theorem~\ref{t3}.

\begin{theorem}\label{Q1}
Let $G$ be a group with $P(G) \cong P(\mathbb{Q})$. Then
$\vec{P}(G) \cong \vec{P}(\mathbb{Q})$.
\end{theorem}

\begin{proof}
Let $x \in \mathbb{Q}$ be non-identity. Then by Lemma~\ref{isomorphism}, we
have that $O(x) \cong I(x)$, so $O(x)' \cong I(x)'$. Let $g \in G$ be
non-identity. Then by Lemma~\ref{Q} we have that $I(g)'$ and $O(g)'$ have no
edges between them. So the complement of the power graph of $G$ restricted to
$N(g)$ consists of two components, one of which is connected. By
Lemma~\ref{isom_nbrs}, an isomorphism $f: P(\mathbb{Q}) \to P(G)$ must
map $I(x)$ to either $I(f(x))$ or $O(f(x))$, since $N(x)'$ and $N(f(x))'$ have
the same number of connected components.  Similarly, $f(O(x))  = O(f(x))$ or
$f(O(x)) = I(f(x))$.

We can now show our result. Suppose that for $x , y \in \mathbb{Q}$, we have
$x \to y$. If $f(x) \to f(y)$, then we claim that
$\vec{P}(G)\cong\vec{P}(\mathbb{Q})$.
We know that $f(O(x)) = O(f(x))$ and similarly $f(I(x)) = I(f(x))$. If
$y \sim z$, then the direction agrees with $f(y) \sim f(z)$, since
$f(x) \in I(f(y))$ implies $f (I(y)) = I ( f(y))$ and similarly for $O(y)$.
Repeating this procedure, we can deduce that the directions of any path
$(f(x)=f(x_1), f(x_2), \ldots , f(x_n))$ agree with those of the corresponding
path $(x = x_1, x_2, \ldots, x_n)$. Since the graph $P(\mathbb{Q})$ is
connected we can reach any point in $P(G)$ by a path starting at $f(x)$.

Now suppose that $f(y) \to f(x)$ instead. Then we have $f(I(x) ) = O(f(x))$
and $f(O(x)) = I(f(x))$. Again, if $y \sim z$, then the direction disagrees
with $f(y) \sim f(z)$, since $f(x) \in O(f(y))$ implies $f(I(y)) = O(f(y))$
and similarly for $O(y)$.  Repeating this procedure, we can deduce that the
directions of any path $(f(x) = f(x_1), f(x_2), \ldots , f(x_n))$ are in exact
reversal with respect to those of the corresponding path
$(x = x_1, x_2, \ldots, x_n)$. Thus $\vec{P}(G)$ has all the arrows reversed
relative to $\vec{P}(\mathbb{Q})$, so we deduce that
$\vec{P}(G)\cong\vec{P}(\mathbb{Q})$. \qed
\end{proof}

We turn now to the group $\mathbb{Q}^n$ for $n>1$, and prove the remaining
part of Theorem~\ref{t3}.

If $a$ and $b$ are non-identity elements, then $a$ and $b$ lie in the same
connected component of the power graph if and only if they span the same
$1$-dimensional vector subspace of $\mathbb{Q}^n$. (For, if $x$ and $y$ lie in
the same vector subspace, then $y=(m/n)x$ for some $m,n\in\mathbb{Z}$, so
$ny=mx$, and there is a path of length $2$ from $x$ to $y$. The converse is
clear.)

So the power graph of
$\mathbb{Q}^n$ consists of countably many disjoint copies of 
$P(\mathbb{Q})\setminus\{0\}$ together with an isolated vertex. For $x\ne0$,
let $Q_x$ denote the connected component containing $x$.

\begin{theorem}
Let $G$ be a group with $P(G) \cong P(\mathbb{Q}^n)$. Then
$\vec{P}(G) \cong \vec{P}(\mathbb{Q}^n)$.
\end{theorem}

\begin{proof}
Let $f: \mathbb{Q}^n\to G$ be a power graph isomorphism. Let $x\in\mathbb{Q}^n$
be non-identity. We deduce from Lemma~\ref{isomorphism}
that $I(x) \cong O(x)$. By Lemma~\ref{order}, $G$ must be torsion free, so we
can apply Lemma~\ref{Q} to deduce, by the same arguments as in the proof of
Theorem~\ref{Q1}, that $O(f(x)) \cong I(f(x))$ (since $f(I(x)) = I(f(x))$ or
$f(I(x)) = O(f(x))$ and similarly for $O(x)$) and thus that the connected
component containing $f(x)$, $C_{f(x)}$, has a directed power graph isomorphic
to that of the connected component $Q_x$ containing $x$. Repeating this
procedure for all the connected components of $G$, we conclude that
$\vec{P}(G) \cong \vec{P}(\mathbb{Q}^n)$. \qed
\end{proof}

\section{Subgroups of $\mathbb{Q}$}

We now examine the power graphs of subgroups of $\mathbb{Q}$. We begin with
a general result.

\begin{lemma}\label{component}
Let $G$ be a nilpotency class $2$ torsion-free group and $C$ a connected
component of $P(G)$. Then the vertices of $C$ form a subgroup of $G$.
\end{lemma}

\begin{proof}
We first show that $x$ and $y$ in $C$ being two steps apart implies that 
$\langle x, y \rangle$ is cyclic. If we have any of the possibilities
\begin{itemize}\itemsep0pt
\item[] $z \to x$, $z \to y$: then $x,y\in\langle z\rangle$;
\item[] $z \to x$, $y \to z$: then $x\in\langle y\rangle$;
\item[] $z \to y$, $x \to z$: similar;
\end{itemize}
then we are done. Hence, suppose that $x \to z$ and $y \to z$, so
$z = x^n = y^m$ for some $n,m\in\mathbb{Z}$. By the same argument as in
Proposition~\ref{torsion}, we see that $\langle x, y \rangle$ is a finitely
generated
abelian group, it must be one of $\mathbb{Z}^2$ or $\mathbb{Z} \times C_k$.
Since $G$ is torsion-free it cannot be $\mathbb{Z}\times C_k$ for $k>1$.
It cannot be $\mathbb{Z}^2$, since $\langle x \rangle$ is a cyclic subgroup of
finite index, contradicting the result in Proposition~\ref{torsion}.
Therefore, $\langle x ,y \rangle \cong \mathbb{Z}$.

Now we show that for all $x, y \in C$, $\langle x , y \rangle$ is cyclic.
We use induction on the length of the path from $x$ to $y$. Suppose that $z$
is the point on the path two steps from $x$. Then $\langle x,z\rangle
=\langle w\rangle$, and the path from $w$ to $y$ is shorter than the path from
$x$ to $y$; so $\langle w,y\rangle$ is cyclic and contains $x$.

Finally, let $x, y \in C$. We have $\langle x, y \rangle = \langle a \rangle$
for some $a \in C$. Then $x  y^{-1} \in \langle a \rangle\subseteq C$, so
$x y^{-1} \in C$. Thus $C \leq G$, as claimed. \qed
\end{proof}

Before we continue, we state a result which can be found in \cite{ref3}.

We define a \emph{unitary subgroup} of $\mathbb{Q}$ to be a subgroup that
contains $1$.

\begin{theorem}\label{Unitary subgroup}
Every non-trivial subgroup $S$ of $\mathbb{Q}$ is isomorphic to at least one
unitary subgroup of $\mathbb{Q}$.
\end{theorem}

Define $P$ to be the set of all prime numbers and let $M$ be the set
$\{f:P \to \{\mathbb{N} \cup \{0,\infty\}\}$. Any $f \in M$ is called a
\emph{height function}.

\paragraph{Definition}
For a unitary subgroup $A$ of $\mathbb{Q}$, the \emph{height function}
$h_A \in M$ associated to $A$ is defined as follows: for each prime $p$,
$h_A(p)= \max\{\alpha:\frac{1}{p^\alpha} \in A\}$.

Next we state two results of \cite{ref3}.

\begin{lemma}\label{Lemma 1}
Let $A$ be a unitary subgroup of $\mathbb{Q}$. Then for relatively prime $m$
and $n$, $m/n \in A$ if, and only if, $1/n \in A$.
\end{lemma}

\begin{lemma}\label{Lemma 2}
Let $A$ be a unitary subgroup of $\mathbb{Q}$. Then for relatively prime $m$
and $n$, $1/(mn) \in A$ if, and only if, $1/m \in A$ and $1/n \in A$.
\end{lemma}

We are now ready to prove several auxiliary results. In what follows, we
work in the power graph $P(A)$ but take restrictions to the set $I_A(x)$ of 
in-neighbours of a vertex in the directed power graph $\vec{P}(A)$.

\begin{lemma}\label{The height funtion}
Let $A$ be a unitary subgroup of $\mathbb{Q}$ and $P(A)$ be the power
graph of $A$. Then there exists $x \in A$ such that the set $I_A(x)$ of
in-neighbours of $x$ in $P(A)$ is infinite if, and only if, either there
exists a prime $p$ such that $h_A(p)=\infty$, or there are infinitely many
primes $q$ such that $h_A(q)>0$.
\end{lemma}

\begin{proof}
If there exists a prime $p$ such that $h_A(p)=\infty$ or there are infinitely
many primes $q$ such that $h_A(q)>0$, then $I_A(1)$ is infinite.

In order to prove the forward implication, we will prove the contrapositive.
Let $x \in A$. If $x=0$, then clearly $I_A(x)$ is finite, so suppose $x\neq 0$.
If $y \in I_A(x)$, then $|y| < |x|$. Let $|y|=m/n$, where $\gcd(m,n)=1$.
Factorize $n$ to the form
\[n= \pm p_1^{\alpha_1}p_2^{\alpha_2} \cdots p_n^{\alpha_n},\] where each
$p_i$ is a prime and $\alpha_i \in \mathbb{N}$. By Lemma~\ref{Lemma 1},
$m/n\in A$ if, and only if, $1/n\in A$. Hence by repeatedly using
Lemma~\ref{Lemma 2}, $m/n\in A$ if, and only if,
\[ \frac{1}{p_i^{\alpha_i}} \in A\]
for all $i \in {1,2,\ldots,n}$. By our assumption, there are only finitely many
numbers in $A$ of the form $\frac{1}{p^{\alpha}}$, where $p$ is a prime and
$\alpha \in \mathbb{N}$, hence there are only finitely many possibilities for
the value of $n$ as $y$ ranges over $I_A(x)$. For fixed $n$, since 
$|m|<|x|\cdot|n|$, there are only finitely many possibilities for  $m$.  We
conclude that $I_A(x)$ is finite. \qed
\end{proof}

\begin{lemma}\label{Cardinalities}
Let $A$ be a unitary subgroup of $\mathbb{Q}$ and $P(A)$ be the power graph of
$A$. Then the following two statements are equivalent:
\begin{itemize}\itemsep0pt
\item There exists $x \in A$ such that $I_A(x)$ is infinite.
\item For all $x \in A$ such that $x \neq 0$, $I_A(x)$ is infinite.
\end{itemize}
\end{lemma}

\begin{proof}
Observe that the converse implication is trivial. So suppose there exists
$x \in A$ such that $I_A(x)$ is infinite. Then by
Lemma~\ref{The height funtion}, there exists a prime $p$ such that
$h_A(p)=\infty$ or there are infinitely many primes $q$ such that $h_A(q)>0$.
Suppose there exists a prime $p$ such that $h_A(p)=\infty$. Let
$y = m/n \in A$ be non-identity. Since $I_A(y) = I_A(-y)$ we can assume without 
loss of generality that $y >0$. Moreover, we can assume $\gcd(m,n)=1$. By
Lemma~\ref{Lemma 1}, $m/n \in A$ if, and only if, $1/n \in A$. If
$\gcd(p,n)=1$, then using Lemma~\ref{Lemma 2}, we have that
$1/(np^{\alpha}) \in A$ for all $\alpha \in \mathbb{N}$. Let
$\gcd(p,m)=p^{\beta}$ and $m=kp^{\beta}$, for some $k \in \mathbb{Z}$.

Since $\gcd(n, k) =1$ and $\gcd(p, k) =1$, we have by Lemma~\ref{Lemma 1},
\[\frac{k}{np^{\alpha}}=\frac{kp^{\beta}}{np^{\alpha +\beta}}=\frac{m}{np^{\alpha +\beta}} \in A\]
for all $\alpha \in \mathbb{N}$.
Hence $I_A(y)$ is infinite. If $\gcd(p,n) \neq 1$, then $\gcd(p,m)=1$.
Factorize $n$ to the form
\[n=  p^{\beta}p_1^{\alpha_1}\cdots p_n^{\alpha_n},\]
where each $p_i$ is a prime and $\alpha_i \in \mathbb{N}$. By repeatedly using
Lemma~\ref{Lemma 2}, $1/n \in A$ if, and only if, $1/p_i^{\alpha_i} \in A$
for all $i \in {1,2,\ldots,n}$ and $1/p^{\beta} \in A$. It follows again by
repeatedly using Lemma~\ref{Lemma 2}, that
\[\frac{1}{ p_1^{\alpha_1}\cdots p_n^{\alpha_n}p^{\alpha}}=\frac{1}{np^{\alpha-\beta}} \in A\]
for all $\alpha \in \mathbb{N}$. Hence as $\gcd(p,m)=1$ and $\gcd(m,n)=1$, we
have
\[\frac{m}{p_1^{\alpha_1}\cdots p_n^{\alpha_n}p^{\alpha}}=\frac{m}{np^{\alpha-\beta}} \in A\]
for all $\alpha \in \mathbb{N}$ with $\alpha > \beta$, and so $I_A(y)$ is
infinite.

Now suppose that there are infinitely many primes $q$ such that
$h_A(q)>0$. As there are only finitely many primes dividing $m$ or $n$, it
follows that there are infinitely many $k \in \mathbb{Z}$ such that
$\gcd(k,n)=1$, $\gcd(k,m)=1$, and $k$ is a product of primes $q$ such that
$h_A(q)>0$. Now it follows by Lemma~\ref{Lemma 1} and Lemma~\ref{Lemma 2}
that for all such $k$ we have $m/kn \in A$. Hence we conclude that $I_A(y)$ is
infinite. \qed
\end{proof}

\begin{lemma}\label{Lemma 3.16}
Let $G$ be an nilpotency class $2$ torsion-free group with a connected power
graph $P(G)$. If $H$ is an nilpotency class $2$ group with $P(G) \cong P(H)$,
then $\vec{P}(G) \cong \vec{P}(H)$.
\end{lemma}

The proof uses the following two results; the first can be found in
\cite{ref1} or \cite[Chapter VIII, Section 30]{ref2}, and the second in
\cite{ref3}.

\begin{prop}
Let $G$ be a group. Then the following two statements are equivalent:
\begin{itemize}\itemsep0pt
\item $G$ is torsion-free and locally cyclic;
\item $G$ is embedded in $\mathbb{Q}$.
\end{itemize}
\end{prop}

\begin{prop}
Let $G$ be a torsion-free group. Then the following two statements are
equivalent:
\begin{itemize}\itemsep0pt
\item $G$ is embedded in $\mathbb{Q}$;
\item for any two non-trivial subgroups $A$ and $B$ of $G$, we have
$A \cap B\neq \{1\}$.
\end{itemize}
\end{prop}

\paragraph{Proof of Lemma~\ref{Lemma 3.16}}
As $P(G) \cong P(H)$, from our previous results it follows that both $H$ and
$G$ are torsion-free and locally cyclic. Hence by Proposition 1 and
Theorem~\ref{Unitary subgroup} we can consider $H$ and $G$ to be unitary
subgroups of $\mathbb{Q}$. Without loss of generality, we can assume $H\neq G$.
By the second Proposition above, we have $G \cap H\neq\{1\}$, and since
$G \cap H$ is torsion-free, it follows that $G \cap H$ contains an infinite
cyclic subgroup of $\mathbb{Q}$.

Let $x \in G \cap H$. Then for all $y \in O(x)$, we have $y \in G \cap H$
(as $G \cap H$ is a subgroup of $\mathbb{Q}$). It follows that for $y \in H$,
if there exists $x \in G \cap H$ such that $x \in I(y)$, then $y \in O(x)$,
and so $y\in G \cap H$. We deduce:

For all $y\in H\setminus G$, if $x\in G\cap H $ and $x\sim y$, then $x\in O(y)$.

For such a $y$, we have $\langle y \rangle\cap(G \cap H)\neq \{1\}$ (by the
second Proposition), so there exists $x \in G \cap H$ such that $x \in O(y)$.
Therefore, using Lemma~\ref{Q} we can recognize $O(y)'$ in $N(y)'$  as the
only connected component that has an element in $G \cap H$. Hence we can
determine $I(y)'$ as well.

Let $f$ be an isomorphism $f: P(G) \to P(H)$. Let $x \in G$ such that 
$x \neq 1_G$. If $f(x) \in H \setminus G$, then by arguments above we can
determine $O(f(x))$ and $I(f(x))$ in $P(H)$. If $f(x) \in G \cap H$, then
$O(f(x)) = \langle f(x) \rangle \leq G \cap H$. However, we know all
directions in $\vec P(G\cap H)$, hence again $O(f(x))$ is determined in $P(H)$, 
and by looking at $N(f(x))'$ in $P(H)'$ and using Lemma~\ref{Q} we can
determine $I(f(x))$ as well.

By Lemma~\ref{Cardinalities}, either $I(x)$ is finite for all $x \in G$ or it
is infinite for all $x \in G \setminus\{1\}$.

Consider now the first case. Let $x \in G$ be such that $x \neq 1_G$. Then
$f(x) \neq 1_H$. It follows by Lemma~\ref{isom_nbrs} that for all $y \in H$,
$I(y)$ is finite. Hence using Lemma~\ref{Q}, as $G$ and $H$ are torsion free,
the only infinite connected components of $N(x)'$ and $N(f(x))'$ are precisely
$O(x)'$ and $O(f(x))'$ respectively. But then again using
Lemma~\ref{isom_nbrs} we deduce $f:I(x) \to I(f(x))$ and $f:O(x) \to O(f(x))$.
As this is true for all $x \in G$ such that $x \neq 1_G$ and $f(1_G) = 1_H$ we
conclude that $f$ induces an isomorphism $f:\vec{P}(G) \to \vec{P}(H)$.

In the second case, we deduce by Lemma~\ref{isom_nbrs} that for all
$y \in H  \setminus\{1_H\}$, $I(y)$ is infinite. We find an isomorphism
$f': \vec{P}(G) \to \vec{P}(H)$. Fix $z_0 \in G$ such that $z_0 \neq 1_G$ and
let $f'(z_0)=f(z_0)$ and $ f'(1_G)=1_H$. Let $z_1 \in G$ such that
$z_1 \sim z_0$ and $z_1 \neq z_0$. By the previous arguments, $O(f'(z_0))$ and
$I(f'(z_0))$ are determined. Hence, if $z_1 \in O(z_0)$, let
$f'(z_1) \in O(f'(z_0))$. Similarly, if $z_1 \in I(z_0)$, let
$f'(z_1) \in I(f'(z_0))$. Finally, if $z_1=z_0^{-1}$, let
$f'(z_1)=f'^{-1}(z_0)$. Then directions of the path $z^1=(z_0,z_1)$ agree
with those of the corresponding path $f'(z)^1=(f'(z_0),f'(z_1))$ and
$f': \{1_G,z_0,z_1\} \to H$ is an injection. We can continue in this manner
to define $f'$ in such a way that it respects the path directions. However, it
remains to show that we can do so in an injective manner.

Thus, let $n \in \mathbb{N}$ and assume that the directions of the path
$(z_0,z_1,\ldots,z_n)$ agree with those of the corresponding path
$(f'(z_0),f'(z_1),\ldots,f'(z_n))$. Furthermore, assume that
\[f': \{1_G,z_0,z_1,\ldots,z_n\} \to H\]
is an injection. Let $z_{n+1} \in G$ be such that $z_{n+1} \sim z_n$ and
$z_{n+1} \neq z_n$. Denote $S:=\{z_i:z_i=z_{n+1}\}$ and
$M:=\{f'(z_i):i \in \{1,2,\ldots,n\}\}$. 

By our results, $O(f'(z_n))$ and $I(f'(z_n))$ are determined. If there exists
$z_i \in S$ such that $z_i=z_{n+1}$, then let $f'(z_{n+1})=f'(z_i)$, also if
$z_{n+1}=z_n^{-1}$, let $f'(z_{n+1})=f'^{-1}(z_n)$. Otherwise, as
$O(f'(z_n))\cup 1_H\cong\mathbb{Z}$ (using the fact that $f'(z_n) \neq 1_H$ and
$H$ is torsion-free), it follows that $O(f'(z_n))$ is infinite. Hence if
$z_{n+1} \in O(z_n)$, we can let $f'(z_{n+1}) \in O(f'(z_n))$ such that
$f'(z_{n+1}) \notin M$. If $z_{n+1} \in I(z_n)$, then by our arguments above
as $f'(z_n) \neq 1_H$, $I(f'(z_n))$ is infinite, therefore we can let
$f'(z_{n+1}) \in I(f'(z_n))$ such that $f'(z_{n+1}) \notin M$. We now have
that the directions of the path $(z_0,z_1,\ldots,z_{n+1})$ agree
with those of the corresponding path $(f'(z_0),f'(z_1),\ldots,f'(z_{n+1}))$.
Furthermore, $f': \{1_G,z_0,z_1,\ldots,z_{n+1}\} \to H$ is an injection. Since
the graph $P(G)$ is connected, we can reach any point in $P(H)$ by a path
starting at $f'(z_0)$. Thus, continuing in this manner we can define $f'$ to
be an isomorphism $f':\vec{P}(G) \to \vec{P}(H)$, as required. \qed

\medskip

Thus, we have the result (Theorem~\ref{t1}) that classifies all torsion-free
nilpotency class 2 groups with respect to their power graphs:

\begin{theorem}
Let $G$ and $H$ be nilpotency class $2$ torsion-free groups. Then
$P(G) \cong P(H)$ implies $\vec{P}(G) \cong \vec{P}(H)$.
\end{theorem}

\begin{proof}
 This is an immediate consequence of Lemma~\ref{component} and
Lemma~\ref{Lemma 3.16}, since an isomorphism from the power graph of $G$ to
the power graph of $H$ is an isomorphism between their connected components.
\end{proof}

It is not true in general that the power graph of a locally cyclic torsion-free
group determines the group up to isomorphism. Before giving a counterexample,
let us introduce new definitions.

For a height function $h$ and a positive integer
$m=p_1^{\alpha_1} \cdots p_n^{\alpha_n}$, where each $p_i$ is a prime and
$\alpha_i \in \mathbb{N}$, define $mh(p)$ as the height function given by
\[mh(p)=\cases{h(p) + \alpha_i & if $p=p_i$,\cr 
h(p) & otherwise.\cr}\]

For two height functions $h$ and $f$ we write $h \equiv f$ if, and only if,
there exist non-negative integers $m$ and $n$ such that $mh=nf$. In other
words, $h\equiv f$ if and only if $h$ and $f$ differ in only finitely many
positions, and they differ only finitely in these positions.

\begin{prop}
The relation $\equiv$ is an equivalence relation.
\end{prop}

The proof is straightforward.

Now the following is shown in \cite{ref3}:

\begin{theorem}\label{Isomorphism of unitary subgroups}
Let $A$ and $B$ be two unitary subgroups of the rationals. Then $A \cong B$ if,
and only if, $h_A \equiv h_B$.
\end{theorem}

Fix a prime $p$ and consider the subgroup of $\mathbb{Q}$, denoted by $G_p$,
generated by all the negative powers of $p$ (it consists of all rational
numbers whose denominator is a power of $p$). Every element of $G_p$ can be
written as a product of powers of primes, where the prime $p$ can have negative
exponent but all the other exponents are non-negative. The height function of
this group has the form
\[h_{G_p}(q)=\cases{\infty & if  $q=p$, \cr 
0 & otherwise.\cr}\]
It follows by Theorem~\ref{Isomorphism of unitary subgroups} that $G_p$ is not isomorphic to $G_q$ whenever $p \neq q$. However, we will show that $\vec{P}(G_p) \cong \vec{P}(G_q)$.

\begin{theorem}
Let $p$ and $q$ be two primes such that $p \neq q$. Let $\pi:P \to P$ be the 
transposition $ \pi = (p, q)$. Then the map $\varphi:G_p \to G_q$ defined by 
\[\pm p_1^{\alpha_1} \cdots p_n^{\alpha_n} 
\mapsto \pm \pi(p_1)^{\alpha_1} \cdots  \pi(p_n)^{\alpha_n}\]
and $0 \mapsto 0$,
induces an isomorphism $\varphi:\vec{P}(G_p) \to \vec{P}(G_q)$.
\end{theorem}

\begin{proof}
Observe that the map $\theta: G_q \to G_p$ given by
\[\pm p_1^{\alpha_1} \cdots p_n^{\alpha_n} 
\mapsto \pm \pi^{-1}(p_1)^{\alpha_1} \cdots  \pi^{-1}(p_n)^{\alpha_n}\]
and $0\mapsto0$,
is an inverse of $\varphi$. Hence $\varphi$ is a bijection. Let $a,b \in G$ 
be such that $a\rightarrow b$ in $\vec{P}(G_p)$. Factorize $a$ and $b$ as
\begin{eqnarray*}
a &=& \pm 2^{\alpha_1}3^{\alpha_2} \cdots p_i^{\alpha_i} \cdots,\\
b &=& \pm 2^{\beta_1}3^{\beta_2} \cdots p_i^{\beta_i} \cdots,
\end{eqnarray*}
where each
$p_i$ is a prime and the exponents are allowed to be zero. By our assumption,
there exists $m \in \mathbb{Z}$ such that $b=ma$. However, this is equivalent
to saying that $\alpha_i \leq \beta_i$ for all $i$. Now as
\begin{eqnarray*}
\varphi(a) &=& \pm \pi(2)^{\alpha_1}\pi(3)^{\alpha_2} \cdots \pi(p_i)^{\alpha_i} \cdots,\\
\varphi(b) &=& \pm \pi(2)^{\beta_1}\pi(3)^{\beta_2} \cdots \pi(p_i)^{\beta_i} \cdots,
\end{eqnarray*}
and $\alpha_i \leq \beta_i$ for all $i$, it follows that
$\varphi(a)\rightarrow \varphi(b)$ in $\vec{P}(G_q)$. This is true for all
such $a,b \in G_p$ and $\varphi(0)=0$. Similarly, $a \to b$ if
$\varphi(a) \to \varphi(b)$. It follows that $\varphi$ induces the required
isomorphism. \qed
\end{proof}

\section{Open problem}

We mention a problem which we have been unable to solve.

If $G$ is a torsion-free nilpotent group of class $2$, and $H$ a group with
$P(G)\cong P(H)$, is it true that $\vec{P}(G) \cong \vec{P}(H)$?

\end{document}